\documentclass[oneside, a4paper,reqno]{amsart}
\usepackage{amscd,amssymb,amsmath,graphicx,verbatim,mathrsfs,xcolor}
\usepackage{wasysym}
\usepackage{hyperref}
\usepackage[TS1,OT1,T1]{fontenc}
\usepackage{lscape} 
\usepackage{fullpage} 
\usepackage[all]{xy}
\newtheorem{theorem}{Theorem}[section]
\newtheorem{lemma}[theorem]{Lemma}
\newtheorem{corollary}[theorem]{Corollary}
\newtheorem{proposition}[theorem]{Proposition}

\theoremstyle{definition}
\newtheorem{definition}[theorem]{Definition}

\newtheorem{example}[theorem]{Example}

\theoremstyle{remark}
\newtheorem{remark}[theorem]{Remark}

\DeclareMathOperator{\Hom}{Hom}

\DeclareMathOperator{\Tor}{Tor} 

\newcommand{\Bcal}{\ensuremath{\mathcal{B}}}

\newcommand{\Ycal}{\ensuremath{\mathcal{Y}}}
\newcommand{\Tcal}{\ensuremath{\mathcal{T}}}

\newcommand{\Ccal}{\ensuremath{\mathcal{C}}}

\newcommand{\Lcal}{\ensuremath{\mathcal{L}}}

\newcommand{\Wcal}{\ensuremath{\mathcal{W}}}

\newcommand{\Susp}{\ensuremath{\mathsf{Susp}}}
\newcommand{\Loc}{\ensuremath{\mathsf{Loc}}}

\newcommand{\Acal}{\ensuremath{\mathcal{A}}}
\newcommand{\Hcal}{\ensuremath{\mathcal{H}}}
\newcommand{\Scal}{\ensuremath{\mathcal{S}}}

\newcommand{\Ucal}{\ensuremath{\mathcal{U}}}
\newcommand{\Vcal}{\ensuremath{\mathcal{V}}}

\newcommand{\add}[1]{\mbox{\rm{add}}(#1)}

\newcommand{\rmod}[1]{\mbox{\rm{Mod}--}{#1}}
\newcommand{\lfmod}[1]{{#1}\mbox{--\rm{mod}}}
\newcommand{\lfproj}[1]{{#1}\mbox{--\rm{proj}}}

\newcommand{\ProjA}{\ensuremath\mbox{\rm{Proj}--$A$}}
\newcommand{\projA}{\ensuremath\mbox{\rm{proj}--$A$}}

\newcommand{\D}{\mathbb{D}}

\newcommand{\Z}{\mathbb{Z}}
\newcommand{\Zbb}{\mathbb{Z}}

\numberwithin{equation}{section}

\begin{document}
\title{Partial silting objects and smashing subcategories}
\author{Lidia Angeleri H\"ugel, Frederik Marks, Jorge Vit{\'o}ria}
\address{Lidia Angeleri H\"ugel, Dipartimento di Informatica - Settore di Matematica, Universit\`a degli Studi di Verona, Strada le Grazie 15 - Ca' Vignal, I-37134 Verona, Italy} \email{lidia.angeleri@univr.it}
\address{Frederik Marks, Institut f\"ur Algebra und Zahlentheorie, Universit\"at Stuttgart, Pfaffenwaldring 57, 70569 Stuttgart, Germany}
\email{marks@mathematik.uni-stuttgart.de}
\address{Jorge Vit\'oria, Department of Mathematics, City, University of London, Northampton Square, London EC1V 0HB, United Kingdom\newline  
\indent Dipartimento di Matematica e Informatica, Universit\`a degli Studi di Cagliari, Pallazzo delle Scienze, via Ospedale, 72, 09124 Cagliari, Italy}
\email{jorge.vitoria@unica.it}


\begin{abstract}
We study smashing subcategories of a triangulated category with coproducts via silting theory. Our main result states that for derived categories of dg modules over a non-positive differential graded ring, every compactly generated localising subcategory is generated by a partial silting object. In particular, every such smashing subcategory admits a silting t-structure. 
\end{abstract}
\maketitle


\section{Introduction}
Smashing subcategories of a triangulated category $\Tcal$ occur naturally as kernels of localisation functors which preserve coproducts. These subcategories have proved to be useful as they often induce a decomposition of $\Tcal$ into smaller triangulated categories, yielding a so-called recollement of triangulated categories in the sense of \cite{BBD}. A typical example of a smashing subcategory is given by a localising subcategory of $\Tcal$, i.e. a triangulated subcategory closed under coproducts, which is generated by a set of compact objects from $\Tcal$. The claim that all smashing subcategories of a compactly generated triangulated category are of this form is sometimes referred to as the Telescope conjecture, first stated for the stable homotopy category in Algebraic Topology (see \cite{Bou, Ra}). Affirmative answers to the conjecture were provided, for example, in \cite{N} and \cite{KS} for derived module categories of commutative noetherian rings and of hereditary rings, respectively. On the other hand, Keller provided an example of a smashing subcategory in the derived category of modules over a (non-noetherian) commutative ring that does not contain a single non-trivial compact object from the ambient derived category (see \cite{Ke}). So, although the Telescope conjecture does not hold in general, it is difficult to determine whether a given triangulated category satisfies it.

\smallskip

In this article, we study smashing subcategories via silting theory. This approach is motivated by some recent work indicating that localisations of categories and rings are intrinsically related to silting objects (see \cite{AMV2,MS,NSZ,PV}). In particular, it was shown in \cite{MS} that universal localisations of rings in the sense of \cite{Scho} are always controlled by a (not necessarily finitely generated) partial silting module. In other words, instead of localising our ring with respect to some set of \textit{small} objects, we can equally localise with respect to a single \textit{large} object that is silting in a suitable way. The main theorem of this article provides a triangulated analogue of this result, realising a  localisation at a set of \textit{compact} objects as a localisation at a single \textit{large} partial silting object.\\

\noindent{\bf Theorem}
{\em
Let $\Tcal$ be the derived category of dg modules over a non-positive differential graded ring and let $\Lcal$ be a localising subcategory of $\Tcal$ generated by a set of compact objects of $\Tcal$. Then there is a partial silting object $T$ in $\Tcal$ such that $\Lcal$ is the smallest localising subcategory containing $T$.}\\

In order to prove this result, we build on some recent work on (partial) silting objects in triangulated categories with coproducts (see \cite{NSZ,PV}). We show that partial silting objects, suitably defined, give rise to smashing subcategories and we study the torsion pairs associated with them. Moreover, for a large class of triangulated categories, we provide a classification of partial silting objects in terms of certain TTF triples, which will be a crucial step towards our main result. As a consequence, we further obtain that all smashing subcategories in the theorem above admit nondegenerate t-structures, which indicates that we should not expect our theorem to hold for arbitrary compactly generated triangulated categories (see e.g. Remark \ref{Remark silting}(2)).

\smallskip

The structure of the paper is as follows. In Section 2, we discuss some preliminaries on localising subcategories and torsion pairs. Section 3 is devoted to the notion of (partial) silting objects and their interplay with smashing subcategories. In particular, Theorem \ref{bij partial silting} provides a classification of (partial) silting objects in terms of their associated torsion pairs. In Section 4, we prove our main theorem.
\medskip

\noindent {\bf Acknowledgements.} 
The authors acknowledge support from the Program Ricerca di Base 2015 of the University of Verona. Lidia Angeleri H\"ugel was partly supported by Istituto Nazionale di Alta Matematica INdAM-GNSAGA. Jorge Vit\'oria acknowledges support from the Engineering and Physical Sciences Research Council of the United Kingdom, grant number EP/N016505/1.


\section{Preliminaries}
By a subcategory of a given category, we always mean a full subcategory closed under isomorphisms.
Throughout, $\Tcal$ will denote a triangulated category admitting arbitrary (set-indexed) coproducts.
Given a subcategory $\Lcal$ of $\Tcal$ and a set of integers $I$, we write $\Lcal^{\perp_I}$ for the intersection $\bigcap_{n\in I}\mathsf{Ker}\Hom_\Tcal(\Lcal,-[n])$. The symbols $>n$, $<n$, $\leq n$ or $\geq n$ will be used to describe the obvious corresponding sets of integers and, if $I=\{0\}$, we simply write $\Lcal^\perp$. If the subcategory $\Lcal$ only consists of an object $X$, we write $X^{\perp_I}$ and $X^\perp$. 
We denote by $\mathsf{Add}(X)$ the subcategory of $\Tcal$ whose objects are summands of coproducts of $X$.
Given two subcategories $\Acal$ and $\Bcal$ of $\Tcal$, we denote by $\Acal\ast\Bcal$ the subcategory of $\Tcal$ consisting of the objects $X$ in $\Tcal$ for which there are objects $A$ in $\Acal$, $B$ in $\Bcal$ and a triangle $$A\longrightarrow X\longrightarrow B\longrightarrow A[1].$$


\subsection{Localising subcategories and torsion pairs}

\begin{definition}
Given a subcategory $\Lcal$ of $\Tcal$ we say that $\Lcal$ is 
\begin{itemize}
\item \textbf{complete} if it is closed under products;
\item \textbf{cocomplete} if it is closed under coproducts;
\item \textbf{localising} if it is a cocomplete triangulated subcategory;
\item \textbf{coreflective} if its inclusion functor admits a right adjoint;
\item \textbf{smashing} if it is a coreflective localising subcategory of $\Tcal$ such that $\Lcal^{\perp}$ is cocomplete.
\end{itemize}
Further, if $\Lcal$ is a localising subcategory and $\Scal$ is a set of objects in $\Tcal$, we say that $\Scal$ {\bf generates} $\Lcal$, if $\Loc(\Scal)=\Lcal$, where $\Loc(\Scal)$ denotes the smallest localising subcategory of $\Tcal$ containing $\Scal$. If all the objects in $\Scal$ are compact in $\Tcal$, i.e. the functor $\Hom_\Tcal(S,-)$ commutes with coproducts for all $S$ in $\Scal$, we say that the localising subcategory $\Lcal$ is {\bf compactly generated}.
\end{definition}

\begin{remark}
More classically, a coreflective localising subcategory $\Lcal$ of $\Tcal$ is called smashing, if the right adjoint to the inclusion functor preserves coproducts. However, having a right adjoint to the inclusion that preserves coproducts is equivalent to having a right adjoint to the quotient functor from $\Tcal$ to $\Tcal/\Lcal$ that preserves coproducts. Since the latter right adjoint identifies $\Tcal/\Lcal$ with $\Lcal^\perp$, we obtain our alternative definition of a smashing subcategory (see, for example, \cite[Definition 3.3.2]{HPS}).
\end{remark}

\begin{definition}
A pair of subcategories $(\Vcal,\Wcal)$ of $\Tcal$ is said to be a \textbf{torsion pair} if 
\begin{enumerate}
\item $\Vcal$ and $\Wcal$ are closed under summands;
\item $\Hom_\Tcal(V,W)=0$ for any $V$ in $\Vcal$ and $W$ in $\Wcal$;
\item $\Vcal\ast \Wcal=\Tcal$.
\end{enumerate}
\end{definition}

It is easy to show that $(\Vcal,\Wcal)$ is a torsion pair if and only if $\Vcal^\perp=\Wcal$, ${}^\perp\Wcal=\Vcal$ and axiom (3) in the above definition holds. In particular, it follows that both classes $\Vcal$ and $\Wcal$ are closed under extensions. Recall that a subcategory of $\Tcal$ is called \textbf{suspended} if it is closed under extensions and positive shifts. 

\begin{definition}\label{def t-structure}
A torsion pair $(\Vcal,\Wcal)$ is said to be
\begin{itemize}
\item a \textbf{t-structure} if $\Vcal$ is suspended,  a \textbf{co-t-structure} if $\Wcal$ is suspended, and a \textbf{stable t-structure} if it is both a t-structure and a co-t-structure;
\item \textbf{left nondegenerate} if $\bigcap_{n\in\mathbb{Z}}\Vcal[n]=0$, \textbf{right nondegenerate} if $\bigcap_{n\in\mathbb{Z}}\Wcal[n]=0$ and \textbf{nondegenerate} if it is both left and right nondegenerate;
\item \textbf{generated by a set of objects $\Scal$} if $\Wcal=\Scal^\perp$.
\end{itemize}
If $(\Vcal,\Wcal)$ is a t-structure, we say that $\Vcal\cap\Wcal[1]$ is its \textbf{heart}. If $(\Vcal,\Wcal)$ is a co-t-structure, we say that $\Vcal[1]\cap\Wcal$ is its \textbf{coheart}. A triple $(\Ucal,\Vcal,\Wcal)$ of subcategories of $\Tcal$ is said to be a \textbf{torsion-torsionfree triple} (or \textbf{TTF triple}) if $(\Ucal,\Vcal)$ and $(\Vcal,\Wcal)$ are torsion pairs in $\Tcal$. We will say that a TTF triple $(\Ucal,\Vcal,\Wcal)$ is \textbf{suspended} if the class $\Vcal$ is suspended and it is called \textbf{stable} if $\Vcal$ is a triangulated subcategory of $\Tcal$. 
\end{definition}

\begin{remark}
Recall that a subcategory of $\Tcal$ is called an \textbf{aisle} if it is suspended and coreflective.
It was shown in \cite{KV} that the assignment $\Vcal\mapsto (\Vcal,\Vcal^\perp)$ yields a bijection between aisles and t-structures in $\Tcal$.
\end{remark}

It is an interesting question whether one can generate torsion pairs from a set of objects $\Scal$ in $\Tcal$. We consider the following two natural candidates
$$(\Ucal_\Scal:={}^\perp(\Scal^{\perp_{\leq 0}}), \Scal^{\perp_{\leq 0}}) \ \ \ \ \ \ \text{and}\ \ \ \ \ \ (\Lcal_\Scal:={}^\perp(\Scal^{\perp_\mathbb{Z}}),\Scal^{\perp_\mathbb{Z}}).$$
If indeed they are torsion pairs, then $(\Ucal_\Scal,\Scal^{\perp_{\leq 0}})$ is the t-structure generated by $\Scal$ and its positive shifts, and $\Ucal_\Scal$ is the smallest aisle containing $\Scal$. On the other hand, $(\Lcal_\Scal,\Scal^{\perp_{\mathbb{Z}}})$ would be the stable t-structure generated by all shifts of $\Scal$, and $\Lcal_\Scal$ would be the smallest coreflective localising subcategory containing the set $\Scal$.
Note that, since $\Scal^{\perp_{\mathbb{Z}}}$ is contained in $\Scal^{\perp_{\leq 0}}$, it follows that  $\Ucal_\Scal$ is contained in $\Lcal_\Scal$.
A natural candidate for $\Ucal_\Scal$ is the smallest cocomplete suspended subcategory containing $\Scal$, which we denote by $\Susp(\Scal)$. On the other hand, a natural candidate for $\Lcal_\Scal$ is $\Loc(\Scal)$, the smallest localising subcategory of $\Tcal$ containing $\Scal$. Indeed, it is easy to show that $\Susp(\Scal)^{\perp}=\Scal^{\perp_{\leq 0}}$ 
and that $\mathsf{Loc}(\Scal)^{\perp}=\Scal^{\perp_\mathbb{Z}}$ and, thus, we always have $\Susp(\Scal)\subseteq \Ucal_\Scal$ and $\Loc(\Scal)\subseteq \Lcal_\Scal$. In the following subsection, we introduce a large class of triangulated categories, for which torsion pairs can be generated as suggested above.


\subsection{Well generated triangulated categories}
We will often consider triangulated categories that have a \textit{nice} set of generators.

\begin{definition}\cite{Neeman0,Krause}
Given a regular cardinal $\alpha$ and $\Tcal$ a triangulated category, we say that 
\begin{itemize}
\item an object $X$ in $\Tcal$ is \textbf{$\alpha$-small} if given any map $h:X\longrightarrow \coprod_{\lambda\in\Lambda} Y_\lambda$ for some family of objects $(Y_\lambda)_{\lambda\in\Lambda}$ in $\Tcal$, the map $h$ factors through a subcoproduct $\coprod_{\omega\in \Omega}Y_\omega$ where $\Omega$ is a subset of $\Lambda$ of cardinality strictly less than $\alpha$;
\item $\Tcal$ is \textbf{$\alpha$-well generated} if it has set-indexed coproducts and it has a set of objects $\Scal$ such that
\begin{itemize}
\item $\Scal^{\perp_{\mathbb{Z}}}=0$;
\item for every set of maps $(g_\lambda:X_\lambda\longrightarrow Y_\lambda)_{\lambda\in \Lambda}$ in $\Tcal$, if $\Hom_\Tcal(S,g_\lambda)$ is surjective for all $\lambda$ in $\Lambda$ and all $S$ in $\Scal$, then $\Hom_\Tcal(S,\coprod_{\lambda\in \Lambda}g_\lambda)$ is surjective for all $S$ in $\Scal$;
\item every object $S$ in $\Scal$ is $\alpha$-small.
\end{itemize}
\item $\Tcal$ is \textbf{well generated} if it is $\alpha$-well generated for some regular cardinal $\alpha$. 
\end{itemize}
\end{definition}

The following are examples of well generated triangulated categories.
\begin{itemize}
\item The derived category of a small differential graded category is $\aleph_0$-well generated or, equivalently, compactly generated (\cite[Subsection 5.3]{Kell}).
\item The homotopy category of projective modules over a ring is $\aleph_1$-well generated (\cite[Theorem~1.1]{Neeman1}). 
\item The derived category of a Grothendieck abelian category $\Acal$ is $\alpha$-well generated, with $\alpha$ a regular cardinal that depends on $\Acal$ (\cite[Theorem 0.2]{Nee0}).
\end{itemize}
An advantage of working with well generated triangulated categories is the following.

\begin{theorem}\cite[Proposition 3.8]{CGR}\cite[Theorem 2.3]{Nee}\label{existence of adjoints}
If $\Tcal$ is a well generated triangulated category and $\Scal\subset\Tcal$ is a set of objects, then both $(\Loc(\Scal),\Scal^{\perp_{\mathbb{Z}}})$ and $(\Susp(\Scal),\Scal^{\perp_{\leq 0}})$ are torsion pairs.
\end{theorem}

\begin{corollary}
Let $\Tcal$ be a well generated triangulated category.
\begin{enumerate}
\item If $\Scal$ is a set of objects in $\Tcal$ with $\Scal^{\perp_\Z}=0$, then we have $\Loc(\Scal)=\Tcal$.
\item If $\Lcal$ is a compactly generated localising subcategory of $\Tcal$, then $\Lcal$ is smashing.
\end{enumerate}
\end{corollary}

\begin{proof}
Statement (1) follows immediately from Theorem \ref{existence of adjoints}.
For $(2)$, it is enough to point out that the localising subcategory $\Lcal=\Loc(\Scal)$, where $\Scal$ is a set of compact objects in $\Tcal$, is coreflective by Theorem \ref{existence of adjoints}. It then follows from the definition of compact object that the subcategory $\Lcal^\perp=\Scal^{\perp_\Z}$ is closed under taking coproducts.
\end{proof}

We will make use of the following characterisation of smashing subcategories.

\begin{proposition}\cite[Propositions 4.2.4, 4.2.5 and 4.4.14]{Nicolas}\label{recollement}
If $\Tcal$ is well generated, then the assignment $\Lcal\mapsto (\Lcal,\Lcal^\perp,(\Lcal^\perp)^\perp)$ yields a bijection between smashing subcategories of $\Tcal$ and stable TTF triples in $\Tcal$. Moreover, such a TTF triple gives rise to a recollement as depicted below, where $j_!$ and $i_*$ denote the inclusions of $\Lcal$ and $\Lcal^\perp$, respectively, into $\Tcal$.
\begin{equation}\nonumber   
\xymatrix@C=0.5cm{\Lcal^{\perp} \ar[rrr]^{i_*} &&& \Tcal \ar[rrr]^{j^*}  \ar @/_1.5pc/[lll]_{i^*}  \ar @/^1.5pc/[lll]_{i^!} &&& \Lcal\ar @/_1.5pc/[lll]_{j_!} \ar @/^1.5pc/[lll]_{j_*} } 
\end{equation}
This assignment defines a further bijection between stable TTF triples in $\Tcal$ and equivalence classes of recollements of $\Tcal$.
\end{proposition}

For the precise definition of recollement, we refer to \cite{BBD}. 
A further property of well generated triangulated categories is related to a representability problem of cohomological functors, i.e.~of functors that send triangles to long exact sequences of abelian groups.

\begin{definition}
We say that our triangulated category $\Tcal$ satisfies \textbf{Brown representability} if every (contravariant) cohomological  functor $\Tcal^{op}\longrightarrow \mathsf{Mod}$-$\mathbb{Z}$ sending coproducts to products is representable. Dually, $\Tcal$ satisfies \textbf{dual Brown representability} if it is complete and every (covariant) cohomological functor $\Tcal\longrightarrow \mathsf{Mod}$-$\mathbb{Z}$ sending products to products is representable.
\end{definition}

Recall that triangulated categories with Brown representability are complete (\cite[Proposition 8.4.6]{Neeman0}). 

\begin{theorem}\cite[Theorem 8.3.3 and Proposition 8.4.2]{Neeman0}
Any well generated triangulated category satisfies Brown representability. Any compactly generated triangulated category satisfies, furthermore, dual Brown representability.
\end{theorem}


\section{Partial silting objects}\label{Section silting}
We start by introducing (partial) silting objects in a triangulated category $\Tcal$ with coproducts. We will study the smashing subcategories associated with them. In a second subsection, we provide a characterisation of (partial) silting objects in terms of certain TTF triples.

\subsection{On the definition of (partial) silting objects}
\begin{definition}\label{def silting}
An object $T$ of $\Tcal$ is said to be 
\begin{enumerate}
\item \textbf{partial silting} if 
\begin{enumerate}
\item[(PS1)] $T$ and its positive shifts generate a t-structure, i.e. $(\Ucal_T:={}^\perp(T^{\perp_{\leq 0}}),T^{\perp_{\leq 0}})$ is a torsion pair;
\item[(PS2)] $\Ucal_T$ is contained in $T^{\perp_{>0}}$;
\item[(PS3)] $T^{\perp_{>0}}$ is cocomplete.
\end{enumerate}
\item \textbf{silting} if $(T^{\perp_{>0}},T^{\perp_{\leq 0}})$ is a t-structure in $\Tcal$.
\end{enumerate}
Two partial silting objects $T$ and $T^\prime$ are said to be \textbf{equivalent} if $\mathsf{Add}(T)=\mathsf{Add}(T^\prime)$.
\end{definition}

\begin{remark}\label{well gen partial}
If $\Tcal$ is a well generated triangulated category, an object $T$ is partial silting if and only if $T$ lies in $T^{\perp_{>0}}$ and $T^{\perp_{>0}}$ is cocomplete (see Theorem \ref{existence of adjoints}).  
\end{remark}

In \cite{NSZ}, partial silting objects were defined as those $T$ satisfying conditions (PS1) and (PS2). Inspired by the definition of partial tilting modules proposed in \cite[Definition 1.4]{CT}, our definition of partial silting objects includes the extra condition (PS3). This will allow us to prove that partial silting objects in well generated triangulated categories generate smashing subcategories. Still, our partial silting objects satisfy the properties proved in \cite{NSZ}, among which we highlight the following.

\begin{lemma}\cite[Theorem 1(2)]{NSZ}\label{lemma NSZ}
If $T$ is an object of $\Tcal$ satisfying (PS1) and (PS2), then 
$T^{\perp_{>0}}=\Ucal_T\ast T^{\perp_{\mathbb{Z}}}.$
\end{lemma}

It is a direct consequence of Lemma \ref{lemma NSZ} that a partial silting object $T$ in $\Tcal$ is silting if and only if $T^{\perp_\Z}=0$.
The above notion of a silting object appeared already in \cite{PV} and \cite{NSZ}. Examples of silting objects are silting complexes in the unbounded derived category of modules over a ring (see \cite{AMV2} and \cite{Wei1}). Notice that also unbounded complexes can occur as silting objects according to our definition, see \cite[Example 7.9]{A} and \cite{AH} for examples.

The following is an example of a partial silting object that is not silting (see also \cite[Subsection 4.4]{Laking}).

\begin{example}
Let $A$ be a (unital) ring and $\Tcal:=\mathbb{K}(\ProjA)$ be the homotopy category of projective right $A$-modules. As an object in $\Tcal$, $A$ is partial silting. Indeed, $A$  obviously   lies in    the subcategory $A^{\perp_{>0}}$, which is cocomplete as it consists of the complexes in $\Tcal$ with vanishing positive cohomologies. Now the claim follows from Remark \ref{well gen partial}
since $\Tcal$ is well generated.

As discussed above, $A$ is silting in $\Tcal$ if and only if $A^{\perp_\Z}=0$. However, if $A$ is a finite dimensional algebra (over a field) with infinite global dimension, then $A^{\perp_\Z}$ is non-trivial. In fact, in this situation $\Tcal$ is compactly generated and, moreover, the functor $\Hom_A(-,A)$ induces an anti-equivalence between the subcategory of compact objects $\Tcal^c$ in $\Tcal$ and the bounded derived category $\mathbb{D}^b(\lfmod{A})$ of finitely generated left $A$-modules (\cite[Theorems 2.4 and 3.2]{Jorgensen}). It then follows from \cite[Theorem 2.1]{Neeman} (see also \cite[Proposition 1.7]{KI}) that the subcategory $A^{\perp_\Z}$ is compactly generated and that $\Tcal^c/\mathbb{K}^b(\projA)$, which is anti-equivalent to $\mathbb{D}^b(\lfmod{A})/\mathbb{K}^b(\lfproj{A})$, embeds fully faithfully into the subcategory of compact objects of $A^{\perp_\Z}$. Here, $\mathbb{K}^b(\projA)$ and $\mathbb{K}^b(\lfproj{A})$ denote the bounded homotopy categories of finitely generated projective right, respectively left, $A$-modules.
Finally, our assumptions on $A$ guarantee that there is a finite dimensional (even simple) left $A$-module $M$ with infinite projective dimension. Then $M$ yields a non-zero object in $\mathbb{D}^b(\lfmod{A})/\mathbb{K}^b(\lfproj{A})$ and, thus, $A^{\perp_{\Z}}$ is non-trivial, as claimed.
\end{example}

The next proposition relates partial silting objects to smashing subcategories.

\begin{proposition}\label{partial silting def}
Let $T$ be an object in a well generated triangulated category $\Tcal$. Then the following are equivalent.
\begin{enumerate}
\item $T$ is partial silting;
\item $\mathsf{Loc}(T)$ is smashing and $T$ is silting in $\mathsf{Loc}(T)$;
\item $T^{\perp_{>0}}$ is an aisle of $\Tcal$ containing $T$.
\end{enumerate}
\end{proposition}
\begin{proof}
If $T$ is partial silting, then $T^{\perp_{\mathbb{Z}}}=\bigcap_{n\in\mathbb{Z}}T^{\perp_{>0}}[n]$ is cocomplete and it follows from Theorem \ref{existence of adjoints} that $\Loc(T)$ is smashing. Moreover, $\Ucal_\Tcal=\Susp(T)$ is contained in $\Loc(T)$, and, therefore, $T$ and its positive shifts generate the t-structure $(\Ucal_T,T^{\perp_{\leq 0}}\cap \Loc(T))$ in $\Loc(T)$. Since, by assumption, $\Ucal_T$ is contained in $T^{\perp_{>0}}\cap\Loc(T)$ and $T^{\perp_{>0}}\cap\Loc(T)$ is cocomplete, it follows that $T$ is partial silting in $\Loc(T)$. But $T$ generates $\Loc(T)$ showing, by Lemma \ref{lemma NSZ}, that $T$ is silting in $\Loc(T)$. Thus, we have (1)$\Rightarrow$(2). 

Let us now prove (2)$\Rightarrow$(3). Consider the recollement induced by $\mathsf{Loc}(T)$, as in Proposition \ref{recollement}. 
Since, by assumption, we have $\Ucal_T=T^{\perp_{>0}}\cap \mathsf{Loc}(T)$, and, by Lemma \ref{lemma NSZ}, we have $T^{\perp_{>0}}=\Ucal_T\ast T^{\perp_{\mathbb{Z}}}$, we see that $T^{\perp_{>0}}$ is precisely the aisle of the t-structure in $\Tcal$ obtained by gluing the t-structures $(T^{\perp_{\mathbb{Z}}},0)$ in $T^{\perp_{\mathbb{Z}}}$ and $(T^{\perp_{>0}}\cap \mathsf{Loc}(T),T^{\perp_{\leq 0}}\cap \mathsf{Loc}(T))$ in $\mathsf{Loc}(T)$ along this recollement (for details on gluing t-structures, we refer to \cite[Th\'eor\`eme 1.4.10]{BBD}). 

The implication (3)$\Rightarrow$(1) follows from Remark \ref{well gen partial}.
\end{proof}

\begin{remark}\label{Remark silting}
\hfill
\begin{enumerate}
\item Note that whenever an object $T$ is silting in a smashing subcategory $\Lcal$ of a well generated triangulated category $\Tcal$, we have $\Lcal=\Loc(T)$. Indeed, $\Loc(T)$ is contained in $\Lcal$, hence $(\mathsf{Loc}(T),T^{\perp_{\mathbb{Z}}}\cap\Lcal)$ is a torsion pair in $\Lcal$. Since $T$ is silting in $\Lcal$, using Lemma \ref{lemma NSZ}, we conclude that the right hand side class of this torsion pair is trivial so that $\Lcal=\Loc(T)$, as wanted.
\item The t-structure associated with a silting object is always nondegenerate (see Definition \ref{def t-structure}). This shows that not all triangulated categories with coproducts admit silting objects, and, in particular, not all smashing subcategories of a given triangulated category $\Tcal$ with coproducts are of the form $\Loc(T)$ for a partial silting object $T$ in $\Tcal$. Take, for example, $\Tcal$ to be the stable category of not necessarily finitely generated modules over a representation-finite self-injective and finite dimensional algebra. By \cite{Dugas}, such an algebra is periodic and, thus, every object in $\Tcal$ is isomorphic to infinitely many of its positive (respectively, negative) shifts. As a consequence, there are no nondegenerate t-structures in $\Tcal$.
\end{enumerate}
\end{remark}


\subsection{Partial silting objects and TTF triples}
In this subsection, we establish a bijection between equivalence classes of partial silting objects and certain TTF triples in $\Tcal$. We will make use of the following lemma.

\begin{lemma}\label{main lemma}
Let $(\Ucal,\Vcal,\Wcal)$ be a suspended TTF triple in $\Tcal$ and let $\Ccal=\Ucal[1]\cap\Vcal$ be the coheart of the co-t-structure $(\Ucal,\Vcal)$. Then the following are equivalent.
\begin{enumerate}
\item $\Ccal^{\perp_{\mathbb{Z}}}=\Ucal^{\perp_{\mathbb{Z}}}$;
\item $\Vcal=\Ccal^{\perp_{>0}}$;
\item the t-structure $(\Vcal,\Wcal)$ is right nondegenerate.
\end{enumerate}
Moreover, if $\Vcal=\Scal^\perp$ for a set of objects $\Scal$ in $\Tcal$, then there is an object $X$ in $\Ccal$ such that $\Ccal=\mathsf{Add}(X)$. 
\end{lemma}

\begin{proof}
(1)$\Rightarrow$(2): It is clear that $\Vcal\subseteq \Ccal^{\perp_{>0}}$. We prove the reverse inclusion. Let $X$ be an object in $\Ccal^{\perp_{>0}}$ and consider the following truncation triangle of $X$  with respect to the t-structure $(\Vcal,\Wcal)$
$$v(X)\longrightarrow X\longrightarrow w(X)\longrightarrow v(X)[1].$$
Since all non-negative shifts of $\Ccal$ are contained in $\Vcal$, we have that $w(X)$ lies in $\Ccal^{\perp_{\leq 0}}$. On the other hand, since both $X$ and $v(X)$ lie in $\Ccal^{\perp_{>0}}$, so does $w(X)$. Therefore, $w(X)$ lies in $\Ccal^{\perp_{\mathbb{Z}}}$ which, by assumption, equals $\Ucal^{\perp_{\mathbb{Z}}}$. Therefore, $w(X)$ is, in particular, an object of $\Vcal$, showing that $w(X)=0$, as wanted.

(2)$\Rightarrow$(3): As above, we have that $\Wcal\subseteq \Ccal^{\perp_{\leq 0}}$. Hence, the intersection $\bigcap_{n\in\mathbb{Z}}\Wcal[n]$ is contained in $\Ccal^{\perp_{\mathbb{Z}}}$ which, by assumption, is also contained in $\Vcal$. Therefore, this intersection consists only of the zero object and the t-structure $(\Vcal,\Wcal)$ is right nondegenerate.

(3)$\Rightarrow$(1): We prove this statement along the lines of \cite[Lemma 4.6]{AMV3}. Since $\Ccal\subseteq \Ucal[1]$, it is clear that $\Ucal^{\perp_{\mathbb{Z}}}\subseteq \Ccal^{\perp_\mathbb{Z}}$. For the reverse inclusion, let $X$ be an object in $\Ccal^{\perp_{\mathbb{Z}}}$ and consider, for each $k$ in $\mathbb{Z}$, the truncation triangle of $X$ with respect to the t-structure $(\Vcal[k],\Wcal[k])$
$$\xymatrix{v^k(X)\ar[r]^{a^k}& X\ar[r]& w^k(X)\ar[r]& v^k(X)[1]},$$
where $v^k(X)$ lies in $\Vcal[k]$ and $w^k(X)$ lies in $\Wcal[k]$. For each $v^k(X)$, consider a truncation triangle for the co-t-structure $(\Ucal[k+1],\Vcal[k+1])$
$$\xymatrix{U_{k+1}^X\ar[r]^{\alpha_k}& v^k(X)\ar[r]^{\beta_k}& V_{k+1}^X\ar[r]& U_{k+1}^X[1]}$$
First, observe that, since both $v^k(X)[1]$ and $V_{k+1}^X$ lie in $\Vcal[k+1]$, so does $U_{k+1}^X[1]$. Hence, we have that $U_{k+1}^X$ lies in $\Ccal[k]$, showing that $a^k\alpha_k=0$, by our assumption on $X$. Thus, there is a map $f_k:V_{k+1}^X\longrightarrow X$ such that $f_k\beta_k=a^k$. However, $f_k$ must also factor through $a^{k+1}\colon v^{k+1}(X)\longrightarrow X$, i.e. there is a map $g_k\colon V^X_{k+1}\longrightarrow v^{k+1}(X)$ such that $a^{k+1}g_k=f_k$. In summary, we get $a^{k+1}g_k\beta_k=a^k$.  Recall that, however, $a^{k+1}$ factors canonically through $a^k$ (since $\Vcal$ is suspended) and, therefore, since $a^k$ is a right minimal approximation, we conclude that $\beta_k$ is a split monomorphism. This shows that $v^k(X)$ lies in $\Vcal[k+1]$ and, therefore, the canonical map $v^{k+1}(X)\longrightarrow v^k(X)$ must be an isomorphism. This implies that also the canonical map $w^k(X)\longrightarrow w^{k+1}(X)$ is an isomorphism, for all $k$ in $\mathbb{Z}$. In particular, for any $k$, the object $w^k(X)$ lies in $\bigcap_{n\in\mathbb{Z}}\Wcal[n]=0$ and, thus, $X$ lies in $\bigcap_{k\in\mathbb{Z}}\Vcal[k]=\Ucal^{\perp_{\mathbb{Z}}}$.

To prove the final assertion, we now assume that $\Vcal=\Scal^\perp$ for a set of objects $\Scal$ in $\Tcal$.
We claim that $\Ccal=\mathsf{Add}(X)$, for some object $X$ in $\Ccal$. Indeed, for each object $S$ in $\Scal$, consider a truncation triangle of $S[1]$ with respect to the co-t-structure $(\Ucal,\Vcal)$
$$\xymatrix{U_S\ar[r]& S[1]\ar[r]^{h_S}& V_S\ar[r]& U_S[1]}.$$
Since $S[1]$ and $U_S[1]$ lie in $\Ucal[1]$, so does $V_S$ and, therefore, $V_S$ lies in $\Ccal$. We consider the object $X:=\coprod_{S\in\Scal}V_S$ and we show that $\mathsf{Add}(X)=\Ccal$. It is clear that $\mathsf{Add}(X)\subseteq \Ccal$. Now, for any object $C$ in $\Ccal$, we consider a right $\mathsf{Add}(X)$-approximation $f\colon X^\prime\longrightarrow C$ and we complete it to a triangle
$$\xymatrix{K\ar[r]&X^\prime\ar[r]^f&C\ar[r]^g&K[1]}.$$
Given $S$ in $\Scal$, since $h_S$ is a left $\Vcal$-approximation and $f$ is a right $\mathsf{Add}(X)$-approximation, we have that $\Hom_\Tcal(S[1],f)$ is surjective. Moreover, we have $\Hom_\Tcal(S[1],X^\prime[1])=0$, and thus we conclude that $\Hom_\Tcal(S[1],K[1])=0$, proving that $K$ lies in $\Vcal$. This shows that $g=0$ and $f$ splits, proving our claim.
\end{proof}

We now prove that partial silting objects in sufficiently nice triangulated categories yield TTF triples. Analogous results for silting complexes in derived categories of rings appeared in \cite[Theorem 4.6]{AMV1}.

\begin{proposition}\label{TTF partial silting}
Let $\Tcal$ be a well generated triangulated category satisfying dual Brown representability. If $T$ is a partial silting object in $\Tcal$, then $({}^\perp(T^{\perp_{>0}}),T^{\perp_{>0}},(T^{\perp_{>0}})^\perp)$ is a suspended TTF triple.
\end{proposition}
\begin{proof}
By Proposition \ref{partial silting def}, there is a t-structure $(T^{\perp_{>0}},\Wcal)$ in $\Tcal$. By \cite[Theorem 3.2.4]{Bondarko} there is a co-t-structure of the form $(\Ucal,T^{\perp_{>0}})$ provided that $T^{\perp_{>0}}$ is complete (which is obvious), the heart of $(T^{\perp_{>0}},\Wcal)$ is an abelian category with enough projectives, and $\Loc(T^{\perp_{>0}})$, the smallest localising subcategory containing $T^{\perp_{>0}}$, satisfies dual Brown representability.

Recall from \cite{BBD} that the heart of any t-structure forms an abelian category. First, we check that the heart $\Hcal=T^{\perp_{>0}}\cap \Wcal[1]$ of the t-structure $(T^{\perp_{>0}},\Wcal)$ has enough projective objects. Indeed, recall that this t-structure is obtained in Proposition \ref{partial silting def} by gluing the trivial t-structure $(T^{\perp_{\mathbb{Z}}},0)$ in $T^{\perp_{\mathbb{Z}}}$ with the silting t-structure $(T^{\perp_{>0}}\cap \mathsf{Loc}(T),T^{\perp_{\leq 0}}\cap\mathsf{Loc}(T))$ in $\Loc(T)$. By \cite[Proposition 4.6]{PV}, the heart $\Acal$ of the latter t-structure has enough projective objects. Since gluing t-structures induces a recollement of hearts (\cite[Section 1.4]{BBD}), and since the heart of the t-structure $(T^{\perp_{\mathbb{Z}}},0)$ is trivial, we conclude that $\Hcal\cong \Acal$ is an abelian category with enough projectives.

Secondly, it suffices to observe that $\mathsf{Loc}(T^{\perp_{>0}})=\Tcal$. Indeed, note that both $\mathsf{Loc}(T)$ and $T^{\perp_{\mathbb{Z}}}$ are contained in $\mathsf{Loc}(T^{\perp_{>0}})$ and, thus, so is $\Tcal=\mathsf{Loc}(T)\ast T^{\perp_{\mathbb{Z}}}$ (see Theorem \ref{existence of adjoints}). This finishes our proof.
\end{proof}

The following theorem establishes the bijection between partial silting objects and certain TTF triples.

\begin{theorem}\label{bij partial silting}
Let $\Tcal$ be a well generated triangulated category satisfying dual Brown representability. The assignment sending an object $T$ to the triple $({}^\perp(T^{\perp_{>0}}),T^{\perp_{>0}},(T^{\perp_{>0}})^\perp)$ yields a a bijection between
\begin{itemize}
\item Equivalence classes of partial silting objects in $\Tcal$;
\item Suspended TTF triples $(\Ucal,\Vcal,\Wcal)$ in $\Tcal$ such that $(\Ucal,\Vcal)$ is generated by a set of objects from $\Tcal$ and $(\Vcal,\Wcal)$ is right nondegenerate.
\end{itemize}
\end{theorem}

\begin{proof}
By Proposition \ref{TTF partial silting}, the triple $({}^\perp(T^{\perp_{>0}}),T^{\perp_{>0}},(T^{\perp_{>0}})^\perp)$ is a suspended TTF triple in $\Tcal$, and the co-t-structure $(\Ucal:={}^\perp(T^{\perp_{>0}}),T^{\perp_{>0}})$ is generated by a set of objects, namely the negative shifts of $T$. It remains to check that the t-structure $(T^{\perp_{>0}},\Wcal:=(T^{\perp_{>0}})^\perp)$
is right nondegenerate to guarantee that the assignment is well-defined. But from the proof of Proposition \ref{TTF partial silting}, we know that $\mathsf{Loc}(T^{\perp_{>0}})=\Tcal$, showing that $\bigcap_{n\in\mathbb{Z}}\Wcal[n]=\Loc(T^{\perp_{>0}})^\perp=\Tcal^\perp=0$, as wanted.

To prove that the assignment is injective, we observe that the coheart $\Ccal=\Ucal[1]\cap T^{\perp_{>0}}$ coincides with $\mathsf{Add}(T)$. Indeed, it is clear that $\mathsf{Add}(T)$ is contained in $\Ccal$. The reverse inclusion follows as in the proof of the last statement of Lemma \ref{main lemma}.
 
Finally, if $(\Ucal,\Vcal,\Wcal)$ is a TTF triple as in the statement above and $\Ccal$ denotes the coheart of the co-t-structure $(\Ucal,\Vcal)$, then it follows from Lemma \ref{main lemma} that $\Vcal=\Ccal^{\perp_{>0}}$ and, moreover, since $(\Ucal,\Vcal)$ is generated by a set of objects, $\Ccal=\mathsf{Add}(X)$ for an object $X$ in $\Tcal$. It remains to show that $X$ is partial silting in $\Tcal$. However, this is clear from Proposition \ref{partial silting def} since $X^{\perp_{>0}}=\Vcal$ is an aisle in $\Tcal$ containing $X$. 
\end{proof}

\begin{corollary}\label{bij silting}
Let $\Tcal$ be a well generated triangulated category satisfying dual Brown representability. The assignment sending an object $T$ to the triple $({}^\perp(T^{\perp_{>0}}),T^{\perp_{>0}},(T^{\perp_{>0}})^\perp)$ yields a a bijection between
\begin{itemize}
\item Equivalence classes of silting objects in $\Tcal$;
\item Suspended TTF triples $(\Ucal,\Vcal,\Wcal)$ in $\Tcal$ such that $(\Ucal,\Vcal)$ is generated by a set of objects from $\Tcal$ and $(\Vcal,\Wcal)$ is nondegenerate.
\end{itemize}
\end{corollary}

\begin{proof}
By Theorem \ref{bij partial silting}, it is enough to check that a partial silting object $T$ in $\Tcal$ is silting if and only if $\bigcap_{n\in\mathbb{Z}}T^{\perp_{>0}}[n]=T^{\perp_{\mathbb{Z}}}=0$. But this follows from Lemma \ref{lemma NSZ}, as discussed in the previous subsection. 
\end{proof}


\section{Compactly generated localising subcategories}
We will prove that, for a large class of triangulated categories $\Tcal$, compactly generated localising subcategories are generated by a partial silting object from $\Tcal$. In Subsection 4.1, we start with the case of localising subcategories of derived module categories which are generated by 2-term complexes of finitely generated projective modules. Here, we can use recent work from \cite{MS} to obtain the statement claimed above. This subsection is meant to serve as a motivation for the general case, which will be discussed afterwards. In fact, in order to prove our main theorem in Subsection \ref{Subsection general case}, we will  use different techniques that mostly rely on the results of Section \ref{Section silting}.


\subsection{The case of 2-term complexes}
In this subsection, $R$ will be a (unital) ring and $\D(R)=\D(\rmod{R})$ will denote the unbounded derived category of all right $R$-modules. Moreover, let $\Sigma$ be a set of maps between finitely generated projective $R$-modules, which we naturally view as a set of compact objects in $\D(R)$. We will be particularly interested in the smashing subcategory $\Loc(\Sigma)$ of $\D(R)$. First, we claim that the subcategory $\Loc(\Sigma)^\perp$ is cohomologically determined. To make this precise, recall that to every set $\Sigma$ like above, we can associate a ring epimorphism $R\longrightarrow R_\Sigma$, called the universal localisation of $R$ at $\Sigma$ (see \cite[Chapter 4]{Scho}). In this case, the category of modules over $R_\Sigma$ identifies via restriction with the full subcategory $\Ycal_\Sigma$ of all $R$-modules $X$ satisfying that $\Hom_R(\sigma,X)$ is an isomorphism for all $\sigma$ in $\Sigma$. The following result was proved in \cite[Proposition 3.3(3)]{CX}, and can also be deduced from the arguments used in the proof of \cite[Lemma 4.2]{AKL1}.

\begin{proposition}\label{prop CX}
With the notation above, we have that $\Loc(\Sigma)^\perp$ coincides with the subcategory of complexes in $\D(R)$ whose cohomologies lie in $\Ycal_\Sigma$.
\end{proposition}

Note that $\Loc(\Sigma)^\perp$ identifies with $\D(R_\Sigma)$ if and only if $\Tor_i^R(R_\Sigma,R_\Sigma)=0$ for all $i>0$ (see, for example, \cite[Proposition 3.6]{CX} and \cite[Theorem 0.7]{NR}).
\medskip

In this subsection, we are interested in complexes that arise from partial silting modules, as introduced in \cite{AMV1}. Recall from \cite[Lemma 4.8(3)]{AMV1} that an $R$-module is partial silting if and only if it is the cokernel of a map $\sigma$ between (not necessarily finitely generated) projective $R$-modules which, when viewed as an object in $\D(R)$, is partial silting in the sense of Definition \ref{def silting}. Hence, $\sigma$ gives rise to the smashing subcategory $\Loc(\sigma)$ of $\D(R)$. As before, we will check that $\Loc(\sigma)^\perp$ is cohomologically determined. To make this precise, recall that to every partial silting module with respect to a map $\sigma$ like above, we can associate a localisation of rings $R\longrightarrow R_\sigma$, called the silting ring epimorphism with respect to $\sigma$ (see \cite{AMV2}). In this case, the category of modules over $R_\sigma$ identifies via restriction with the full subcategory $\Ycal_\sigma$ of all $R$-modules $X$ such that $\Hom_R(\sigma,X)$ is an isomorphism. We obtain an analogous statement to Proposition \ref{prop CX}.

\begin{proposition}
With the notation above, we have that $\Loc(\sigma)^\perp$ coincides with the subcategory of complexes in $\D(R)$ whose cohomologies lie in $\Ycal_\sigma$.
\end{proposition}

\begin{proof}
The statement follows analogously to Proposition \ref{prop CX}. For the reader's convenience, we include an argument based on the proof of \cite[Lemma 4.2]{AKL1}. Note that an object $Y$ lies in $\Loc(\sigma)^\perp$ if and only if, for all $n\in\Zbb$, we have $\Hom_{\D(R)}(\sigma,Y[n])=0$. If we write $\sigma\colon P\longrightarrow P'$ as a map of projective $R$-modules, this translates to asking that $\sigma$ induces an isomorphism of the form 
$$\xymatrix{\Hom_{\D(R)}(P',Y[n])\ar[rr]&& \Hom_{\D(R)}(P,Y[n])}$$ for all $n\in\Z$. Since, for any projective $R$-module $P$ and any $Y$ in $\D(R)$, we have $\Hom_{\D(R)}(P,Y)\cong\Hom_R(P,H^0(Y))$, it follows that $Y$ lies in $\Loc(\sigma)^\perp$ if and only if $H^n(Y)$ lies in $\Ycal_\sigma$ for all $n\in\Zbb$.
\end{proof}

Keeping in mind the discussion above, the main result of this subsection comes as a direct consequence of the fact that every universal localisation is a silting ring epimorphism (see \cite[Theorem 6.7]{MS}).

\begin{theorem}
Let $\Sigma$ be a set of maps between finitely generated projective $R$-modules. Then there is a partial silting $R$-module with respect to a map $\sigma$ such that $\Loc(\Sigma)=\Loc(\sigma)$ in $\mathbb{D}(R)$.
\end{theorem}


\subsection{The general case}\label{Subsection general case}
In this subsection, $\Tcal$ will denote again a triangulated category with coproducts and we will focus on partial silting objects as defined in Definition \ref{def silting}.
We begin by recalling a result, which states that, for many triangulated categories, sets of compact objects give rise to suspended TTF triples. 

\begin{proposition}\label{cg TTF}
Suppose that $\Tcal$ satisfies Brown representability and let $\Sigma$ be a set of compact objects in $\Tcal$. Then there is a suspended TTF triple of the form $({}^\perp(\Sigma^{\perp_{>0}}),\Sigma^{\perp_{>0}},(\Sigma^{\perp_{>0}})^\perp)$. 
\end{proposition}

\begin{proof}
Since $\Sigma$ is a set of compact objects,  the subcategory $\Sigma^{\perp_{>0}}$ is cocomplete and it follows from \cite[Theorem 4.3]{AI} that the set of all negative shifts of objects from $\Sigma$ generates a torsion pair, namely $({}^\perp(\Sigma^{\perp_{>0}}),\Sigma^{\perp_{>0}})$. As $\Tcal$ satisfies Brown representability, it follows from \cite[Theorem 3.1.2]{Bondarko} that we have a TTF triple, as wanted. 
\end{proof}

In view of Theorem \ref{bij partial silting}, it is often enough to check that the t-structure $(\Sigma^{\perp_{>0}},(\Sigma^{\perp_{>0}})^\perp)$ is right nondegenerate to guarantee that a TTF triple as in Proposition \ref{cg TTF} arises from a partial silting object $T$ in $\Tcal$. This, however, cannot always be the case. Take, for example, $\Sigma$ to be the set of all compact objects in a compactly generated triangulated category $\Tcal$. Then we have $\Sigma^{\perp_{>0}}=\Sigma^{\perp_\Z}=0$ and, thus, $(\Sigma^{\perp_{>0}})^\perp=\Tcal$. The following proposition provides some sufficient conditions for $\Sigma$ to induce a silting TTF triple.

\begin{proposition}\label{cg vs partial}
Let $\Sigma$ be a set of compact objects in a well generated triangulated category $\Tcal$ satisfying dual Brown representability. Assume that for any $\sigma$ in $\Sigma$ there is an integer $n_\sigma>0$ for which $\sigma[n_\sigma]$ lies in $\Sigma^{\perp_{>0}}$. Then there is a partial silting object $T$ in $\Tcal$ such that $\Sigma^{\perp_{>0}}=T^{\perp_{>0}}$. In particular, we have that $\mathsf{Loc}(\Sigma)=\mathsf{Loc}(T)$.
\end{proposition}

\begin{proof}
We denote by $({}^\perp(\Sigma^{\perp_{>0}}),\Sigma^{\perp_{>0}},\Wcal)$ the suspended TTF triple obtained from Proposition \ref{cg TTF}. To show that this TTF triple arises from a partial silting object $T$ in $\Tcal$, following Theorem \ref{bij partial silting}, we only need to check that $\bigcap_{n\in\mathbb{Z}}\Wcal[n]=0$. Note that for the latter it will be sufficient to check that $\Loc(\Sigma^{\perp_{>0}})=\Tcal$. By our assumption on $\Sigma$, it is clear that $\mathsf{Loc}(\Sigma^{\perp_{>0}})$ contains $\mathsf{Loc}(\Sigma)$. Since it also contains $\Sigma^{\perp_{\mathbb{Z}}}$, and since $\Tcal=\mathsf{Loc}(\Sigma)\ast \Sigma^{\perp_{\mathbb{Z}}}$, it follows that $\mathsf{Loc}(\Sigma^{\perp_{>0}})=\Tcal$, as wanted. 
For the final assertion, note that, since $\Sigma^{\perp_{>0}}=T^{\perp_{>0}}$, we also have 
$$\Sigma^{\perp_{\mathbb{Z}}}=\bigcap_{n\in\mathbb{Z}}\Sigma^{\perp_{>0}}[n]=\bigcap_{n\in\mathbb{Z}}T^{\perp_{>0}}[n]=T^{\perp_{\mathbb{Z}}}.$$ 
This implies that $\mathsf{Loc}(\Sigma)=\mathsf{Loc}(T)$, as wanted.
\end{proof}

To show that all compactly generated localising subcategories of a given triangulated category have a partial silting generator, we should ask the underlying triangulated category to be {\em large enough} in a suitable sense (compare with Remark \ref{Remark silting}). One way of guaranteeing this, is to ask for the existence of a compact silting object. Note that if $\Tcal$ is an algebraic triangulated category, i.e. $\Tcal$ is triangle equivalent to the stable category of a Frobenius exact category, the existence of a compact silting object precisely says that $\Tcal$ is triangle equivalent to the derived category of dg modules over a non-positive differential graded ring (see \cite[Theorem 8.5]{Keller} and \cite{Kell}). Thus, we restrict ourselves to this setting.

\begin{theorem}
Let $\D(R)$ be the derived category of dg modules over a non-positive differential graded ring $R$ and let $\Lcal$ be a compactly generated localising subcategory of $\D(R)$. Then there is a partial silting object $T$ in $\D(R)$ such that $\Lcal=\Loc(T)$. In particular, every compactly generated localising subcategory of $\D(R)$ admits a nondegenerate t-structure.
\end{theorem}

\begin{proof}
By assumption, $R$ is a compact silting object in $\D(R)$. By \cite[Proposition 8.2]{Keller}, the compact objects in $\D(R)$ are precisely those in the smallest thick subcategory containing $R$. Therefore, we can assume, without loss of generality, that the localising subcategory $\Lcal$ is generated by a set of compact objects $\Sigma$ in $\D(R)$ such that every $\sigma$ in $\Sigma$ is a finite iterated extension of non-positive shifts of objects from $\add R$, where $\add R$ denotes the subcategory of $\D(R)$ whose objects are summands of finite direct sums of copies of $R$. Since $R$ is non-positive we have that $\Hom_{\D(R)}(R,R[i])=0$ for all $i>0$, from which it follows that for any $\sigma$ in $\Sigma$ there is an integer $n_\sigma>0$ for which $\sigma[n_\sigma]$ lies in $\Sigma^{\perp_{>0}}$. Hence, Proposition \ref{cg vs partial} completes the argument. Finally, the existence of a nondegenerate t-structure in $\Lcal$ is a consequence of Proposition \ref{partial silting def}.
\end{proof}

The following example shows that possessing a nondegenerate t-structure does not imply that a given smashing subcategory $\Lcal$ of $\mathbb{D}(R)$ is compactly generated. 
Nevertheless, it is not clear whether the existence of a silting object in $\Lcal$ guarantees compact generation.

\begin{example}
A typical example of a smashing subcategory that is not compactly generated can be obtained by considering an idempotent maximal ideal $I$ of a discrete valuation ring $R$. Then $\mathsf{Loc}(I)$ inside $\D(R)$ is known to be smashing but not compactly generated (see \cite[Theorem 7.2]{BaSt}). Moreover, since $f\colon R\longrightarrow R/I$ is a homological ring epimorphism, the derived category $\mathbb{D}(R/I)$ identifies with the subcategory $\Loc(I)^\perp$ of $\mathbb{D}(R)$ via the restriction functor $f_*$. 
Now, it is clear that the standard t-structure in $\mathbb{D}(R)$ restricts to $\mathbb{D}(R/I)$ (by intersection) and that, moreover, the heart of the restriction, namely $\rmod{R/I}$, identifies with a Serre subcategory of $\rmod{R}$. From \cite[Lemma 3.3]{CR} it then follows that the standard t-structure in $\mathbb{D}(R)$ is obtained by gluing t-structures along the recollement induced by $\mathsf{Loc}(I)$. In particular, $\mathsf{Loc}(I)$ admits a nondegenerate t-structure.
\end{example}


\end{document}